\documentclass{article}
\usepackage[utf8]{inputenc}
\usepackage{tikz-cd}
\usetikzlibrary{cd}
\usepackage{amsmath,amssymb,amsthm,amsfonts,amscd, url,mathtools}
\usepackage{mathrsfs}
\usepackage{tikz}
\usetikzlibrary{calc}
\usepackage{zref-savepos}
\usepackage{tikz}
\usepackage{graphicx}
\graphicspath{ {images/} } 
\usepackage[english]{babel}
\usepackage[autostyle]{csquotes}

\DeclareFontFamily{U}{MnSymbolC}{}
\DeclareSymbolFont{MnSyC}{U}{MnSymbolC}{m}{n}
\DeclareFontShape{U}{MnSymbolC}{m}{n}{
    <-6>  MnSymbolC5
   <6-7>  MnSymbolC6
   <7-8>  MnSymbolC7
   <8-9>  MnSymbolC8
   <9-10> MnSymbolC9
  <10-12> MnSymbolC10
  <12->   MnSymbolC12}{}
\DeclareMathSymbol{\intprod}{\mathbin}{MnSyC}{'270}

\newtheoremstyle{problemstyle}
{\topsep} % Space above
{20pt} % Space below
{} % Body font
{} % Indent amount
{\scshape} % Theorem head font
{\newline} % Punctuatnion after theorem head
{1em} % Space after theorem head
{} % Theorem head spec (can be left empty, meaning `normal')

\newtheorem*{fact}{Fact}

\newtheorem{defs}{Definition}
\newtheorem{prop}{Proposition}
\newtheorem{egs}{Example}
\newtheorem{cor}{Corollary}
\newtheorem{lem}{Lemma}
\newtheorem{thm}{Theorem}
\theoremstyle{problemstyle}
\newtheorem*{rem}{Remark}

\newcommand{\supp}{\text{supp}}

\newcommand{\id}{\text{Id}}

\newcommand{\mL}{\mathcal{L}}

\newcommand{\mH}{\mathcal{H}}
\newcommand{\ra}{\rightarrow}

\newcommand{\bR}{\mathbb{R}}
\newcommand{\bC}{\mathbb{C}}
\newcommand{\bZ}{\mathbb{Z}}

\newcommand{\bN}{\mathbb{N}}

\newcommand{\msU}{\mathscr{U}}

\newcommand{\mfX}{\mathfrak{X}}

\newcommand{\Hom}{\text{Hom}}
\newcommand{\End}{\text{End}}

\newcommand{\bge}{\begin{equation*}}
\newcommand{\ene}{\end{equation*}}

\newcommand{\bra}{\langle}
\newcommand{\ket}{\rangle}

\title{Kostant-Souriau Prequantization}
\author{Ethan Ross }
\date{Summer 2019}

\begin{document}

\maketitle
\section{Introduction}
Our universe is a quantum one. This can be seen in experiments, like Tonomura's 1989 electron double-slit experiment \cite{allday}. The set up of this experiment is an electron emitter, a detector screen, and a panel with two slits in between the emitter and the screen. A beam of electrons is then fired from the emitter to the detector screen after passing through the two slits. As expected, each electron leaves a point impression on the screen, but the resulting density pattern is that of a wave interfering with itself. This result cannot be explained by classical means. 

Why we even study classical mechanics beyond historical or mathematical interest? For the physicist, the answer lies in the fact that classical mechanics provides models for producing quantum theories. For instance, the Schr\"{o}dinger equation was \enquote{derived} in formal analogy with the Hamiltonian formulation of classical mechanics.

This brings us to the subject of this report: a method by which we can take a classical system and \enquote{quantize} it to obtain the corresponding quantum system. The key insight is due to Dirac \cite{dirac}, and is encoded in his famous Dirac axioms. In a nutshell, he noticed that the observables of a classical system and the observables of a quantum system have key formal similarities. Hence, a scheme for producing a quantum system from a classical one should respect these properties.

However, as shown by Groenwald and Van Hove, if the Dirac axioms are taken to be the definition of  a quantization, then no quantization can exist. See Abraham and Marsden \cite{marsden} for a proof of this fact. Rather than giving up, we weaken the axioms and study what's called a prequantization, the first step towards a quantization scheme.

The prequantization scheme reviewed in this report is due to  Kostant \cite{kostant} and Souriau \cite{souriau}. As a summary, if a symplectic manifold $(M,\omega)$ represents the state space of a classical system, then a Kostant-Souriau prequantization is a line bundle over $M$ with some extra structure compatible with the symplectic form $\omega$. It provides two things:
\begin{itemize}
    \item[(i)] a Hilbert space, the state space of a quantum system, and
    \item[(ii)] a map from the classical observables $C^\infty(M)$ to the quantum observables on $\mH$ which respects their formal properties.
\end{itemize}
The existence of such a line bundle turns out to be equivalent to an integrality condition on the symplectic form $\omega$, as the following theorem shows.
\begin{thm} (Kostant-Souriau)

 A Kostant-Souriau prequantization of a symplectic manifold $(M,\omega)$ exists if and only if $[\omega]$ is integral. 
\end{thm}

As the name would suggest, Kostant-Souriau prequantization does not provide an actual quantization of a classical system. In some sense, the Hilbert space it constructs \enquote{depends on too many variables,} which the Heisenberg uncertainty principle \cite{bran} explicitly forbids. Polarizations, which \enquote{cut down on half the variables,} are one solution, but they introduce their own technical difficulties. Metaplectic corrections and half-forms are then brought in to fix the issues arising from polarizations, but these half-forms don't always exist. The interested reader should consult Bates and Weinstein \cite{sean} or Guillemin and Sternberg \cite{guil} for more details.

This report is divided into three sections. In the first, the basic ideas from classical and quantum physics will be discussed, which will then lead us to a description of Kostant and Souriau's construction. The second section will delve into the theory of complex line bundles and the various structures associated to them, such as covariant derivatives, connections, and local systems. In the final section, a brief review of relevant ideas from \v{C}ech cohomology will be provided, paving the way to the proof of Theorem 1. 

\newpage
\subsection{Notation}
\begin{itemize}
    \item $C^\infty(M,N)$ - smooth maps between manifolds $M$ and $N$.
    \item $C^\infty(M)=C^\infty(M,\bR)$
    \item $\Gamma(U,A)$ - sections of a fibre bundle with total space $A$ and $U$ an open subset of base space. 
    \item $TM$ - tangent bundle of smooth manifold.
    \item $\mfX(M)$ - vector fields on a manifold.
    \item $\mfX_\bC(M)$ - complex vector fields on a manifold.
    \item $\Omega^p(M)$ - differential $p$-forms on a manifold $M$.
    \item $H^p(M)$ - $p$-th De Rham cohomology group of $M$.
    \item $\check{H}^p(\msU;R)$ - $p$-th \v{C}ech cohomology group with respect to a cover $\msU$ with coefficients in $R$.
    \item $S\mH$ - essentially self-adjoint operators on a Hilbert space $\mH$.
\end{itemize}
\newpage
\tableofcontents
\newpage
\section{Prequantization}
In this first section, we give a brief tour of the ideas from physics that are relevant to geometric prequantization. In particular, we review the symplectic formulation of classical mechanics, and the Schr\"{o}dinger picture of non-relativistic quantum mechanics. We finish the review of physics with motivation for prequantization by example, then leap into the construction of Kostant and Souriau introduced in the previous section.

\subsection{Geometric Mechanics}
We begin with a review of the relevant formalism from geometric classical mechanics. The interested reader can consult Lee \cite{lee}, Arnold \cite{arnold}, or da Silva \cite{ana} for a more detailed treatment.

If $(M,\omega)$ is a symplectic manifold, then since $\omega$ is non-degenerate, we have the map
\bge
\widetilde{\omega}:TM\ra T^*M;\quad v\mapsto \omega(v,-),
\ene
where $T^*M$ denotes the cotagent bundle, is an isomorphism of vector bundles. Using this fact, we can then define Hamiltonian vector fields.

\begin{defs}
Let $(M,\omega)$ be a symplectic manifold. If $f\in C^\infty(M)$, then define the Hamiltonian vector field of $f$, denoted $X_f$ by
\bge
df=\omega(X_f,-).
\ene
\end{defs}

\begin{defs}
A Hamiltonian system is a triple $(M,\omega,H)$, where $(M,\omega)$ is a symplectic manifold, called the phase space, and $H\in C^\infty(M)$, called the Hamiltonian. We call integral curves of $X_H$ trajectories and $C^\infty(M)$ observables.
\end{defs}

\begin{rem}
In this way of stating classical mechanics, we are thinking of $M$ as representing all possible positions and momenta of a system. The level sets of $H$ correspond to allowable energies of trajectories, and the observables correspond to the outcome of a measurement. 
\end{rem}

The observables of a Hamiltonian system $(M,\omega,H)$ can clearly be given a commutative algebra structure via pointwise multiplication. There is however another multiplication structure that we can give $C^\infty(M)$, called the Poisson bracket.

\begin{defs}
Let $f,g\in C^\infty(M)$. Define their Poisson bracket $\{f,g\}\in C^\infty(M)$ by
\bge
\{f,g\}=\omega(X_f,X_g).
\ene
\end{defs}

\begin{prop}
Let $(M,\omega)$ be a symplectic manifold. Then the Poisson bracket \bge
\{\cdot,\cdot\}:C^\infty(M)\times C^\infty(M)\ra C^\infty(M);\quad (f,g)\mapsto  \{f,g\}
\ene
is a Lie bracket. Furthermore, if $f,g\in C^\infty(M)$, then 
\bge
[X_f,X_g]=X_{\{f,g\}}.
\ene
\end{prop}

\begin{proof}
Left to the reader.
\end{proof}

The Poisson bracket is valuable since it allows us to compute ''time derivatives" of observables along a trajectory. What I mean by that is if $\Phi_t$ denotes the flow of $X_H$ and $f\in C^\infty(M)$, then $f\circ \Phi_t$ is a measurement of whatever quantity $f$ represents along the trajectory. Lie derivatives are used to show 
\bge
\frac{d}{dt} (f\circ \Phi_t)=\{f,H\}\circ\Phi_t.
\ene
See \cite{ana} for a proof.

We get a very useful corollary from this result 
\begin{cor}
$f\in C^\infty(M)$ is constant along every trajectory of the Hamiltonian system $(M,\omega,H)$ $\iff$ $\{f,H\}=0$. 
\end{cor}

Let's finish off this introduction to geometric mechanics with an example.

\begin{egs}\normalfont \textit{One particle moving in $\bR^3$ subject to a conservative force.}
\newline

Let $F:\bR^3\ra\bR^3$ denote the force and $m$ the mass. If $q:\bR\ra\bR^3$ denotes the trajectory the particle takes when acted on by the force $F$, then Newton tells us that $q$ must satisfy
\bge
F=m\frac{d^2 q}{dt^2},
\ene
a second order ODE. Thus, we are solving for both $q$ and its first derivative. Hence, we choose the manifold which parameterizes this system to be $M=\bR^3\times\bR^3$. Give $M$ linear coordinates $(q,p)=(q_1,q_2,q_3,p_1,p_2,p_3)$. We think of the $p_i$ as the coordinates for the momentum of the particle.

For the symplectic form, choose
\bge
\omega=\sum_{j=1}^3 dq_i\wedge dp_i.
\ene
It's an easy exercise to show that $\omega$ is indeed a symplectic form. Now we need our Hamiltonian. 

A force $F:\bR^3\ra \bR^3$ is said to be conservative if there exists $U:\bR^3\ra\bR$, called the potential energy function of $F$, such that
\bge
F=-\mathrm{grad} \ U.
\ene
Since $U$ has units of energy and the Hamiltonian is supposed to represent the energy of a trajectory, we choose Hamiltonian
\bge
H(q,p)=\frac{|p|^2}{2m}+U(q),
\ene
where $\frac{|p|^2}{2m}$ represents the kinetic energy of the trajectory. This is clearly smooth, hence $(M,\omega,H)$ is a Hamiltonian system. To see that trajectories of this Hamiltonian system correspond to solutions of Newton's equation, we need to first compute $X_H$.

For any $f\in C^\infty(\bR^3\times \bR^3)$, one can show that 
\bge
X_f=\sum_{j=1}^3\bigg(\frac{\partial f}{\partial p_j}\frac{\partial }{\partial q_j}-\frac{\partial f}{\partial q_j}\frac{\partial }{\partial p_j}\bigg)
\ene
and hence a trajectory $\gamma(t)=(q(t),p(t))$ must satisfy
\begin{align*}
(X_H)_{\gamma(t)}&=\sum_{j=1}^3\bigg(\frac{\partial H}{\partial p_j}\frac{\partial }{\partial q_j}-\frac{\partial H}{\partial q_j}\frac{\partial }{\partial p_j}\bigg)\\
&=\gamma'(t)\\
&=\sum_{j=1}^3\bigg(\frac{d q_j}{dt}\frac{\partial }{\partial q_j}+\frac{dp_j}{dt}\frac{\partial }{\partial p_j}\bigg).
\end{align*}
Comparing coefficients, we see that $\gamma$ satisfies the famous Hamilton's canonical equations
\bge
\begin{cases}
\frac{dq_j}{dt}&=\frac{\partial H}{\partial p_j}\\
\frac{dp_j}{dt}&=-\frac{\partial H}{\partial q_j}.
\end{cases}
\ene
In our case, this means the trajectory satisfies
\bge
\begin{cases}
\frac{dq_j}{dt}&=\frac{p_j}{m}\\
\frac{dp_j}{dt}&=-\frac{\partial U}{\partial q_j}
\end{cases}
\ene
Putting these equations together, we must have
\bge
m\frac{d^2q}{dt^2}=-\mathrm{grad} \ U(q)=F
\ene
so we've returned to Newton's equation.
\end{egs}

\begin{rem}
We see here why the classical picture cannot provide an adequate description of the double-slit experiment for an electron. If the electron is a particle, then we expect the screen to record a bunch of dots all focused around the path of the beam. We don't see this. So if the electron is a wave, then we expect a continuous interference pattern. We also don't see this. There is some wave-particle duality of the electron which this formulation is incapable of expressing. 
\end{rem}

\subsection{Schr\"{o}dinger Picture of Quantum Mechanics}
To set up the Schr\"{o}dinger picture of quantum mechanics, we need a list of data that is remarkably similar to that defining a classical mechanical system. We need some object to represent all admissable states a system can be in and some observable that controls the evolution of the system. In classical mechanics, this role was played by the symplectic manifold and the Hamiltonian. In quantum mechanics, this role is played by a Hilbert space and an essentially self-adjoint operator, also called the Hamiltonian. As this review skips most of the important features of quantum mechanics, refer to Folland \cite{folland} or Bransden and Joachain \cite{bran} for more details.

Just for the sake of comparison, I introduce some nonstandard terminology.

\begin{defs}
A quantum Hamiltonian system is a pair $(\mH,H)$, where $\mH$ is a Hilbert space and $H$ is an essentially self-adjoint operator on $\mH$. We call $\mH$ the state space and $H$ the Hamiltonian operator.
\end{defs}

We think of elements of $\mH$ as representing all possible configurations that can occur in a given mechanical system. However, there is an arbitrariness to a state that can't be measured called phase. So, in fact, we say two elements $v,w\in\mH$ represent the same state if there exists $\lambda\in\bC^*$ such that $v=\lambda w$.

We have the states of a quantum system, what are the observables? I.e. what objects represent the outcomes of experiments? It turns out \cite{folland} that observables on a quantum Hamiltonian system can be best interpreted as essentially self-adjoint operators. The set of all such operators is denoted $S\mH$. If $T\in S\mH$ represents some dynamical quantity we wish to measure, then we think of the eigenvalues of $T$ as the the outcomes of said measurement.

The last element of this formalism we need to introduce is dynamics. We will make use of the Schr\"{o}dinger picture to do this.

\begin{defs}
Let $(\mH,H)$ be a quantum Hamiltonian system. A wave function is a map $\Psi:\bR\ra\mH$ such that
\begin{itemize}
    \item[(i)] $\bra \Psi(t),\Psi(t)\ket=1$ for all $t\in\bR$.
    \item[(ii)] $\Psi$ satisfyies the Schr\"{o}dinger equation
    \bge
    i\frac{d}{dt}\Psi(t)=H\Psi(t).
    \ene
\end{itemize}
\end{defs}
A wavefunction $\Psi$ is supposed to represent all properties that an evolving mechanical system can posses for all times. Thus, (i) can be interpreted to mean that the state has probability 1 of existing for all times.

\begin{rem}
A quantum Hamiltonian system is analogous to a classical one. We still think of $H$ as representing the energy that a particular wave function carries with it. The only difference here is that we think of the eigenvalues $\{E_n\}_{n\in\bN}$ of $H$ as the allowed energies, hence the discrete \enquote{quantum} nature of this formulation. Furthermore, the Schr\"{o}dinger equation for a quantum Hamiltonian system is the quantum analog of the classical formula
\bge
\gamma'(t)=(X_H)_{\gamma(t)}
\ene
for integral curves of the Hamiltonian $H$ of a Hamiltonian system $(M,\omega,H)$
\end{rem}

We noted that the observables in the classical case had a Lie algebra structure given by the Poisson bracket. In the quantum case, the operators don't even carry a vector space structure since their domains of definition may not be the same. But if $T:D_T\ra \mH$ and $S:D_S\ra\mH$ are two observables satisfying $D_S=D_T$ and $T(D_T)\subset D_T$, $S(D_S)\subset D_S$, then expressions of the form
\bge
T+S,\quad T\circ S-S\circ T=[T,S]
\ene
make sense. We say $T$ and $S$ are composable if these expressions are defined and essentially self-adjoint.

In the first formulations of quantum mechanics, the observables were always taken to be bounded self-adjoint operators. In this case, the set of quantum observables always carries a Lie algebra structure. However, as we shall see in the coming example, the Hamiltonian of a free particle in $\bR^3$ is not even bounded! 

\begin{egs}\normalfont
\textit{A particle in $\bR^3$ with mass $m$ subject to a conservative force $F$.}
\newline

Since our particle is moving through $\bR^3$ and, as we said before, a wavefunction represents the states of the system through time, the natural Hilbert space is $L^2(\bR^3)$ with the usual Lebesgue measure. Hence, a wave function $\Psi$ is of the form
\bge
\Psi(t)=\psi(q,t),
\ene
where $q=(q_1,q_2,q_3)$ is the coordinate on $\bR^3$.

The Hamiltonian which experimentally works here is
\bge
H=-\frac{1}{2m}\nabla^2+U(q),
\ene
where $\nabla^2$ denotes the Laplacian. Note that $H$ has domain of definition all $f\in L^2(\bR^3)$ with square integrable first and second distributional derivatives. 

Since $\Psi$ satisfies the Schrodinger equation, we have
\bge
i\frac{\partial \psi}{\partial t}=-\frac{1}{2m} \nabla^2\psi+U(q)\psi.
\ene

Note that if we restrict to the case where no force is acting on the particle, then the Schr\"{o}dinger equation reduces to
\bge
i\frac{\partial \psi}{\partial t}= \nabla^2\psi.
\ene
It turns out, all solutions to this equation are of the form
\bge
\psi(r,t)=Ce^{i(q,t)\cdot E}
\ene
where $C\in\bR$ and $E\in\bR^4$ are constants and $\cdot$ denotes the usual scalar product in $\bR^4$. Notice that this is the equation of a standing wave evolving through time.  I say here that this wave-function is not normalized in the sense of Definition 5. This isn't a problem if instead of taking $\bR^3$ as the domain of our wavefunction for each time, we take some box of side length $L$.
\end{egs}

\begin{rem}
The formula we obtained for a free particle moving through $\bR^3$ goes a long way to explaining the outcome of the double-slit experiment discussed in the introduction. It explains the interference pattern on the screen and it explains why the interference pattern is composed of discrete dots.

In that experiment, there are four parts of the trajectory of the electron. 
\begin{itemize}
    \item[(i)] The electron moves freely through space after being fired at the screen.
    \item[(ii)] It interacts with the double-slit.
    \item[(iii)] The particle moves freely through space once again.
    \item[(iv)] Finally it hits the detector screen. 
\end{itemize}
So if we model the electron as a free particle in $\bR^3$ before and after it interacts with the screen, then indeed we will expect an interference pattern on the detector screen. 

So now there's the matter of the discrete dots. If the electron is a wave, then classically we expect a continuous interference pattern, not a discrete one. This is taken care of by what we said a quantum observable does: it returns its eigenvalues after an experiment. In our case, if we give $\bR^3$ coordinates $(q_1,q_2,q_3)$ and suppose that the screen lies in the $(q_1,q_2)-plane$, then if $\psi(q,t)$ is the wavefunction of the electron, the numbers which the screen is recording are 
\bge
\int_\bR^3 q_1|\psi(q,t_0)|^2dq\quad\text{and}\quad \int_\bR^3 q_2|\psi(q,t_0)|^2dq
\ene
where $t_0$ is the time when the electron hits the screen. Since $\psi(q,t)$ represents a wave which is interfering with itself, we get higher magnitude numbers when the wave constructively interferes with itself, and lower magnitude numbers when the wave destructively interferes with itself. These numbers represent the probability that the electron will have a given $q_1$ or $q_2$ coordinate at time $t_0$. Hence, since we are firing many electrons at a screen, we expect the pattern of measurements (the dots where the electrons hit the detector) to be discrete and mimic a wave interference pattern.
\end{rem}

\subsection{Dirac Quantization}
A quantization is a proceedure for taking a classical system and turning it into a quantum one. To get a feel for the Dirac axioms, let's see what we need to do to turn example 1 into example 2.

In the classical case, we have a Hamiltonian system $(M,\omega,H)$ with 
\begin{itemize}
    \item[(i)] manifold $M=\bR^6$ with coordiates $(q,p)=(q_1,q_2,q_3,p_1,p_2,p_3)$,
    \item[(ii)] symplectic form
    \bge
    \omega=\sum_{j=1}^3 dq_j\wedge dp_j,
    \ene
    \item[(iii)] and Hamiltonian
    \bge
    H_{class}=\frac{|p|^2}{2m}+U(q).
    \ene
\end{itemize}

In the quantum case, we have a quantum Hamiltonian system $(\mH,H)$, with
\begin{itemize}
    \item[(i)] Hilbert space $\mH=L^2(\bR^3)$,
    \item[(ii)] and Hamiltonian 
    \bge
    H_{quant}=-\frac{1}{2m}\nabla^2+U(q)
    \ene
\end{itemize}

We want a map $Q:C^\infty(\bR^3)\ra S\mH$ such that $Q(H_{class})=H_{quant}$. If we impose $Q$ is linear, then
\bge
Q\bigg(\frac{|p|^2}{2m}+U(q)\bigg)=\frac{1}{2m}\sum_{j=1}^3 Q(p_j^2)+Q(U(q))
\ene
In order for this to return the Schr\"{o}dinger equation discussed in example 2 for any choice of $U(q)$, we impose
\bge
\begin{cases}
Q(1)&=1\\
Q(q_i)&=q_i\\
Q(p_j)&=i\frac{\partial}{\partial q_j}.
\end{cases}
\ene

Note that in the classical case
\bge
\{q_i,p_j\}=\delta_{ij}
\ene
and that in the quantum case we have
\bge
\bigg[q_i,i\frac{\partial}{\partial q_j}\bigg]=-i\delta_{ij}.
\ene
Thus, by definition of $Q$, we have
\bge
Q(\{q_i,p_j\})=-i[Q(q_i),Q(p_j)].
\ene

This motivates the axioms Dirac gave for a prequantization
\begin{defs}
Let $(M,\omega)$ be a symplectic manifold. A prequantization of $(M,\omega)$ is a pair $(Q,\mH)$, where $\mH$ is a Hilbert space and $Q:C^\infty(M)\ra S\mH$ is a map such that
\begin{itemize}
    \item[(i)] $Q(f)$ and $Q(g)$ are composable for all $f,g\in C^\infty(M)$.
    \item[(ii)] $Q$ is linear.
    \item[(iii)] If $1\in C^\infty(M)$ is the constant function $1$, then $Q(1)=2\pi I$, where $I:\mH\ra\mH$ is the identity.
    \item[(iv)] $Q(\{f,g\})=-i[Q(f),Q(g)]$
\end{itemize}
\end{defs}

\begin{rem}
One thing to note about the axioms as I've presented them is that I've introduced the factor of $2\pi$ in axiom (iii). This is to keep in line with the theory of prequantization as presented by Kostant \cite{kostant}. It can be done away with, but we keep the $2\pi$ since it cleans up the statements of many of the theorems. 
\end{rem}

\subsection{Kostant-Souriau Prequantization}
Let $(M,\omega)$ be a symplectic manifold with $\dim M=2n$. Our goal this section is to define a prequantum line bundle over $M$, then demonstrate how this induces a prequantization. The next section will then deal with when such a prequantum line bundle exists. We will mostly be following \cite{lerman} and \cite{sean}.

The first step is to define a Hilbert space from the manifold. For some motivation, consider the following example.

\begin{egs}
Since $\omega$ is non-degenerate, $\omega^n$ is a volume form. So a natural choice for a Hilbert space would be the $L^2$ completion of
\bge
\{f\in C^\infty(M,\bC) \ | \ \supp f\text{ is compact}\}
\ene
with respect to the inner product
\bge
(f,g):=\int_M f\overline{g}\omega^n.
\ene
Notice that $C^\infty(M,\bC)$ can be identified with the space of sections of the trivial complex line bundle $M\times\bC\ra M$.
\end{egs}

So we see that we can obtain a Hilbert space from the trivial bundle on a symplectic manifold. This is due to the fact that the trivial bundle automatically comes equipped with a fibre-wise Hermitian metric. So more generally, we will want to consider all line bundles with Hermitian metrics over our manifold in order to have a large class of candidates for pre-quantization. 

\begin{defs}
A line bundle over $M$ is fibre bundle $L\xrightarrow{\pi} M$ such that 
\begin{itemize}
    \item[(i)] For all $x\in M$, $L_x:=\pi^{-1}(x)$ has the structure of a 1-dimensional complex vector space.
    \item[(ii)] For all $x\in M$, there exists open set $U\subset M$ about $x$ and a diffeomorphism
    \bge
    \psi:\pi^{-1}(U)\ra U\times\bC
    \ene
    such that for all $y\in U$, $\psi|_{L_y}:L_y\ra \{y\}\times\bC$ is a linear isomorphism.
\end{itemize}
\end{defs}

\begin{defs}
Let $U\subset M$ be an open subset. A local section of $L\ra M$ is a smooth function $s:U\ra L$ such that $\pi\circ s=\id_U$. Let $\Gamma(U,L)$ denote all such local section. Define $\Gamma(L):=\Gamma(M,L)$ and call such sections global. Let $\Gamma_c(L)$ denote compactly supported global sections.
\end{defs}

As in the example, we want to use the fact that $\omega$ is non-degenerate to produce an inner product on compactly supported sections. To do this, we need a Hermitian form.

\begin{defs}
A Hermitian structure on $L\ra M$ is a section $\bra,\ket\in\Gamma(L^*\otimes L^*)$ such that for all $x\in M$, $\bra,\ket_x:L_x\times L_x\ra \bC$ is a Hermitian inner product. Denote a line bundle with a Hermitian structure by the pair $(L,\bra,\ket)$ and call it a Hermitian line bundle.
\end{defs}

We can give $\Gamma_c(L)$ a pre-Hilbert space structure as follows. If $s_1,s_2\in\Gamma_c(L)$, then $\bra s_1,s_2\ket:M\ra \bC$ is a smooth compactly supported function. Hence, is integrable. Define
\bge
(s_1,s_2):=\int_M\bra s_1,s_2,\ket \omega^n.
\ene

It's easy to show the following fact.

\begin{fact}
$(\Gamma_c(L),(,))$ is a pre-Hilbert space.
\end{fact}

Let $\mH$ denote the $L^2$ completion of $\Gamma_c(L)$. Our goal now is to associate an unbounded self-adjoint operator on $\mH$ to each smooth function on $M$.  To motivate the construction, let's return to the trivial bundle.

\begin{egs}
Let $f\in C^\infty(M)$ and consider the trivial bundle $M\times\bC\ra M$. We want to make $f$ into a linear operator on $C^\infty(M,\bC)$. One way to do this is to define
\bge
m_f:C^\infty(M,\bC)\ra C^\infty(M,\bC);\quad g\mapsto fg.
\ene
This is trivially linear. 

The other way to make $f$ into a linear operator on $C^\infty(M,\bC)$ makes use of the symplectic structure on $M$. Let $X_f\in\mfX(M)$ be the Hamiltonian vector field of $f$. Note that the map
\bge
C^\infty(M)\ra C^\infty(M);\quad g\mapsto X_fg.
\ene
is a real linear map. Extending $\bC$-linearly, $X_f$ is then a linear operator on $C^\infty(M,\bC)$. 
\end{egs}

Given a line bundle over $M$, we want to somehow lift the two operators defined in the previous example to act on arbitrary sections. To do this, we need the concept of the covariant derivative. 

\begin{defs}
A covariant derivative $\nabla$ is a rule such that for all open $U\subset M$ there is a $\bC$ bilinear map
\bge
\nabla_U:\mfX_\bC(U)\times \Gamma(U,L)\ra \Gamma(U,L);\quad (\xi,s)\mapsto \nabla_\xi s,
\ene
where $\mfX_\bC(U)$ denotes local complex vector fields on $U$, such that
\begin{itemize}
    \item[(i)] If $U\subset V$, $\xi\in\mfX_\bC(V)$ and $s\in \Gamma(V,L)$, then
    \bge
    \bigg(\nabla_{\xi}s\bigg)\bigg|_{U}=\nabla_{\xi|_U}(s|_U).
    \ene
    \item[(ii)] If $f\in C^\infty(U,\bC)$, $\xi\in\mfX_\bC(U)$, and $s\in\Gamma(U,L)$, then
    \begin{align*}
    \nabla_{f\xi}s&=f\nabla_\xi s\\
    \nabla_\xi (fs)&=(\xi f)s+f\nabla_\xi s.
    \end{align*}
\end{itemize}
\end{defs}

Combining the maps from the example with a covariant derivative and adding some normalization factors, we can define the prequantization map.

\begin{defs}
Let $f\in C^\infty(M)$ be a smooth function, $X_f\in\mfX(M)$ denote its Hamiltonian vector field. Define $Q_f\in \Hom(\Gamma(L),\Gamma(L))$ by
\bge
Q_f:=i \nabla_{X_f}+2\pi m_f,
\ene
where $m_f$ is the multiplication operator introduced in example 4. We call the map
\bge
Q:C^\infty(M)\ra\Hom(\Gamma(L),\Gamma(L)); \quad f\mapsto Q_f
\ene
the pre-quantization map.
\end{defs}

\begin{rem}
If $f\in C^\infty(M)$ is constant, then $X_f=0$. Hence, $\nabla_{X_f}=0$. So, if $f=1$, we then obtain $Q_1=2\pi m_1=2\pi \id$. Thus, the map $Q$ satisfies one of the axioms of a prequantization right away.
\end{rem}

\begin{defs}
Let $(L,\bra,\ket)$ be a Hermitian line bundle over $M$, and $\nabla$ a covariant derivative on $L$. We say $\nabla$ and $\bra,\ket$ are compatible if for all open $U\subset M$, $\xi\in\mfX_\bC(U)$, and $s_1,s_1\in\Gamma(U,L)$
\bge
\xi\bra s_1,s_2\ket=\bra \nabla_\xi s_1,s_2\ket+\bra s_1,\nabla_\xi s_2\ket.
\ene
Denote by $(L,\bra,\ket,\nabla)$ a Hermitian line bundle with compatible covariant derivative.
\end{defs}

\begin{lem}
Let $(L,\bra,\ket,\nabla)$ be a line bundle over $M$ with compatible covariant derivative. Then for all $s_1,s_2\in\Gamma_c(L)$ and $f\in C^\infty(M)$ we have
\bge
(Q_fs_2,s_1)=(s_1,Q_fs_2).
\ene
\end{lem}

\begin{proof}
Trivially,
\bge
(fs_1,s_2)=(s_1,fs_2).
\ene
Next, since $\nabla$ and $\bra,\ket$ are compatible, we have
\bge
\int_M\bra \nabla_{X_f}s_1,s_2\ket \omega^n=\int_M X_f\bra s_1,s_2\ket\omega^n-\int_M\bra s_1,\nabla_{X_f}s_2\ket\omega^n.
\ene

Now, since $\mL_{X_f}\omega=0$, we have $\mL_{X_f}\omega^n=0$. Hence,
\bge
\mL_{X_f}(\bra s_1,s_2\ket \omega^n)=\mL_{X_f}(\bra s_1,s_2\ket)\omega^n+\bra s_1,s_2\ket \mL_{X_f}\omega^n=X_f\bra s_1,s_2\ket \omega^n.
\ene
On the otherhand, since $\bra s_1,s_2\ket \omega^n$ is a top form, we have
\bge
\mL_{X_f}(\bra s_1,s_2\ket \omega^n)=X_f\intprod d(\bra s_1,s_2\ket \omega^n)+d(X_f\intprod\bra s_1,s_2\ket\omega^n)=d(X_f\intprod\bra s_1,s_2\ket \omega^n).
\ene
Thus, by Stoke's theorem
\begin{align*}
\int_M X_f\bra s_1,s_2\ket\omega^n=\int_Md(X_f\intprod\bra s_1,s_2\ket \omega^n)=0.
\end{align*}
Hence,
\bge
(\nabla_{X_f}s_1,s_2)=-(s_1,\nabla_{X_f}s_2).
\ene
Putting everything together,
\begin{align*}
(Q_fs_1,s_2)&=([i\nabla_{X_f}+2\pi f]s_1,s_2)\\
&=i(\nabla_{X_f}s_1,s_2)+(2\pi fs_1,s_2)\\
&=-i(s_1,\nabla_{X_f}s_2)+(s_1,2\pi fs_2)\\
&=(s_1,Q_fs_2).
\end{align*}
\end{proof}

All that's left for us is to obtain the relation
\bge
Q_{\{f,g\}}=-i[Q_f,Q_g].
\ene
This is a condition on the curvature of the covariant derivative on $L$

\begin{defs}
Let $\nabla$ be a covariant derivative on $L\ra M$. Define the curvature of $\nabla$, $R\in \Omega^2(M)\otimes\End(L)$ by
\bge
R(\xi,\eta)s=[\nabla_\xi,\nabla_\eta]s-\nabla_{[\xi,\eta]}s,
\ene
where $\xi,\eta$ are local complex vector fields and $s$ is a local section of $L$.
\end{defs}

We now impose the Bohr-Sommerfeld quantization condition on $(M,\omega)$. We demand that $2\pi i\omega$ is the curvature of $\nabla$. That is, for all local complex vector fields $\xi$ and $\eta$ on $M$:
\begin{equation}
[\nabla_\xi,\nabla_\eta]-\nabla_{[\xi,\eta]}=2\pi i\omega(\xi,\eta).
\end{equation}

\begin{lem}
If equation (1) condition holds, then the map
\bge
Q:C^\infty(M)\ra \End(\Gamma_c(L),\Gamma_c(L));\quad f\mapsto Q_f
\ene
satisfies
\bge
Q_{\{f,g}\}=-i[Q_f,Q_g].
\ene
\end{lem}

\begin{proof}
Let $f,g\in C^\infty(M)$. Then by equation (1),
\begin{align*}
[\nabla_{X_f},\nabla_{X_g}]&=\nabla_{[X_f,X_g]}+2\pi i\omega(X_f,X_g)\\
&=\nabla_{X_{\{f,g\}}}+2\pi i\{f,g\}.
\end{align*}

Fix a local section $s$. We compute
\begin{align*}
Q_f(Q_gs)&=Q_f(i\nabla_{X_g}s+2\pi gs)\\
&=-\nabla_{X_f}\nabla_{X_g}s+2\pi i\nabla_{X_f}(gs)+2\pi if\nabla_{X_g}s-4\pi^2fgs\\
&=-\nabla_{X_f}\nabla_{X_g}s+2\pi i\{f,g\}s+2\pi i[f\nabla_{X_g}+g\nabla_{X_f}]s+fgs.
\end{align*}
Thus, swapping $f,g$ we compute
\begin{align*}
Q_f(Q_gs)-Q_g(Q_fs)&=-[\nabla_{X_f},\nabla_{X_g}]s+4\pi i\{f,g\}s\\
&=-\nabla_{X_{\{f,g\}}}s-2\pi i\{f,g\}s+4\pi i\{f,g\}s\\
&=iQ_{\{f,g\}}.
\end{align*}
\end{proof}

\begin{defs}
Let $(M,\omega)$ be a symplectic manifold. A Kostant-Souriau prequantum line bundle over $(M,\omega)$ is a Hermitian line bundle $(L,\bra,\ket,\nabla)$ over $M$ with compatible covariant derivative with curvature $R^\nabla=2\pi i\omega$.
\end{defs}

We have thus shown the following result.
\begin{thm}
A Kostant-Souriau prequantum line bundle $(L,\bra,\ket,\nabla)$ over a symplectic manifold $(M,\omega)$ induces a prequantization with Hilbert space the $L^2$ completion of $\Gamma_c(L)$ and prequantization map
\bge
Q:C^\infty(M)\ra \Hom(\Gamma_c(L),\Gamma_c(L));\quad f\mapsto i\nabla_{X_f}+2\pi m_f.
\ene
\end{thm}

The savvy reader should now notice that we assumed a lot in order to get to this point. For the symplectic manifold $(M,\omega)$ and line bundle $L\ra M$ we assumed
\begin{itemize}
    \item[(i)] There exists a Hermitian structure on $L$.
    \item[(ii)] There exists a comptaible covariant derivative.
    \item[(iii)] The Kostant-Souriau prequantization condition holds.
\end{itemize}

Turns out, these are all conditions on the cohomology class of $[\omega]$. We will dedicate the rest of the report to showing this fact.

\begin{rem}
The Kostant-Souriau prequantization gives us other important facts for free. First, the prequantization map $Q$ is actually injective. Second, this prequantization carries with it a projective unitary representation of the symplectomorphisms on $(M,\omega)$. To learn more, consult \cite{sean} or \cite{lerman}.
\end{rem}

\newpage

\section{Line Bundles and Covariant Derivatives}
\subsection{$\bC^*$-Principal Bundles And Connection 1-Forms}

The first topic we will be discussing on our way to proving the main theorem will be $\bC^*$-principal bundles and connection 1-forms. We will use this material to obtain a convenient way of describing covariant derivatives on line bundles which we will use heavily in the final proof. For a more complete introduction to the topics discussed in this subsection in a broader context, consult \cite{koba} or \cite{greub}.

$\bC^*$-principal bundles are a special kind of fibre bundle with a compatible action by $\bC^*$ on the total space. Before we define $\bC^*$-principal bundles in full generality, let's see the trivial example.

\begin{egs}
Let $M$ be a manifold and define $P=M\times\bC^*$. With the natural projection $pr_1:P\ra M$, we see that $P\xrightarrow{pr_1} M$ is a trivial fibre bundle with typical fibre $\bC^*$. This bundle has a smooth action by $\bC^*$: Let $x\in M$, $\mu,\lambda\in\bC^*$. Define
\bge
(x,\mu)\cdot\lambda:=(x,\mu\lambda).
\ene
This is trivially a smooth group action and is called the trivial action.

Note that the trivial action permutes elements of the fibre over $x$. 
\end{egs}

Taking the above example as a local model, we define a $\bC^*$-principal bundle in full generality.

\begin{defs}
A $\bC^*$-principal bundle is a fibre bundle $P\xrightarrow{\sigma} M$ with typical fibre $\bC^*$ together with a smooth free action
\bge
P\times \bC^*\ra P;\quad (p,\lambda)\mapsto p\cdot \lambda
\ene
such that
\begin{itemize}
    \item[(i)] If $p\in P$, then $q\in\sigma^{-1}(\sigma(p))$ $\iff$ there exists $\lambda\in\bC^*$ such that $q=p\cdot \lambda$.
    \item[(ii)] For all $x\in M$, there exists open neighbourhood $U\subset M$ and a diffeomorphism
    \bge
    \psi:\pi^{-1}(U)\ra U\times\bC^*
    \ene
    such that the diagram commutes
    \bge
    \begin{tikzcd}
    \sigma^{-1}(U)\arrow[rr,"\psi"]\arrow[dr,"\sigma"] & & U\times\bC^*\arrow[dl,"pr_1"]\\
    & U &
    \end{tikzcd}
    \ene
    And such that for all $p\in \pi^{-1}(U)$ and $\lambda\in\bC^*$
    \bge
    \psi(p\cdot\lambda)=\psi(p)\cdot\lambda
    \ene
    where we give $U\times\bC^*$ the trivial action. Call such a $(U,\psi)$ a local trivialization of $P\ra M$.
\end{itemize}
\end{defs}

\begin{rem}
Let $\lambda\in\bC^*$. For ease of notation, we will identify $\lambda$ with the smooth function
\bge
P\ra P;\quad p\mapsto p\cdot \lambda.
\ene
\end{rem}

\begin{egs}
As we've seen a trivial bundle $M\times\bC^*\ra M$ is a $\bC^*$-principal bundle. 
\end{egs}

\begin{egs}
For a less trivial example, let $L\xrightarrow{\pi}M$ be a line bundle. For all $x\in M$, define
\bge
L^+_x:=L_x\setminus\{0_x\}
\ene
where $0_x\in L_x$ denotes the zero element over $x$. We can then define the frame bundle of $L\ra M$ to be 
\bge
L^+:=\bigsqcup_{x\in M}L_x^+
\ene
and give it a projection $\widetilde{\pi}:=\pi|_{L^+}$.

Note that $L^+\subset L$ is a submanifold. If we let $(U,\psi)$ be a local trivialization of $L\ra M$, then since $\psi|_{L_x}$ is a linear isomorphism for all $x\in U$, we obtain that
\bge
\psi|_{\widetilde{\pi}^{-1}(U)}:\widetilde{\pi}^{-1}(U)\ra U\times\bC^*
\ene
is a diffeomorphism which commutes with the projections $\widetilde{\pi}$ and $pr_1$.

Observe that $L^+$ carries a natural action by $\bC^*$ which is smooth and free. Hence $L^+\xrightarrow{\pi} M$ is a $\bC^*$-principal bundle.
\end{egs}

Turns out, the above example classifies all principal $\bC^*$-principal bundles, but we will not need that fact in this report.

For the rest of this subsection, fix a principal $\bC^*$-bundle $P\xrightarrow{\sigma}M$ over $M$. We will now see how the Lie algebra of $\bC^*$, $\bC$, interacts with the fibre bundle structure.

\begin{defs}
 For every $C\in\bC$, define the fundamental vector field on $P$ associated to $C$, $X_C\in\mfX(P)$, by
\bge
(X_C)_p:=\frac{d}{dt}\bigg|_{t=0} p\cdot e^{2\pi i Ct}.
\ene
\end{defs}

Here's a sequence of facts that we will be using, but will not be proving as they would constitute too large of a tangent from the core material.
\begin{prop}
\begin{itemize}
    \item[(i)] The map
    \bge
    \bC\ra \mfX(P);\quad C\mapsto X_C
    \ene
    is a Lie algebra homomorphism.
    \item[(ii)] For all $C\in\bC$, $X_C=CX_1$. 
    \item[(iii)] $X_1$ never vanishes and hence defines a subbundle $VP$ of $TP$, called the vertical bundle.
    \item[(iv)] For all $p\in P$
    \bge
    d\sigma_p=V_pP.
    \ene
\end{itemize}
\end{prop}

\begin{proof}
See \cite{greub}.
\end{proof}

So a principal bundle always has a special subbundle of $TP$ specified by the action of $\bC^*$. This attaches a copy of $\bC$ to every point $p\in P$. Is there a way to attach a copy of $T_{\sigma(p)}M$ to $p$ as well in a transverse way? The answer is yes, and a formula on how to do this is provided by a connection 1-form.

\begin{defs}
 A connection 1-form is a 1-form $\varphi\in\Omega^1(P)\otimes\bC$ such that
\begin{itemize}
    \item[(i)] $X_C\intprod \varphi=C$ for all $C\in\bC$.
    \item[(ii)] For all $\lambda\in\bC^*$ we have $\lambda^*\varphi=\varphi$.
\end{itemize}
\end{defs}

\begin{egs}
Let $P=M\times\bC^*$. Give $\bC^*$ coordinate $z$. Then for any $A\in\Omega^1(M)\otimes\bC$, define
\bge
\varphi=\bigg(A,\frac{1}{2\pi i}\frac{dz}{z}\bigg).
\ene

First, let's compute $X_1$ on $P$. For any $(x,\mu)\in M\times\bC^*$, we have
\bge
(X_1)_{(x,\mu)}=\frac{d}{dt}\bigg|_{t=0}(x,\mu)\cdot e^{2\pi it}=\frac{d}{dt}\bigg|_{t=0}(x,\mu e^{2\pi it})=(0_x, 2\pi i \mu),
\ene
where $0_x\in T_xM$ denotes the zero element above $x$. 

Now we have to show $\varphi(X_1)=1$. We compute,
\begin{align*}
\varphi_{(x,\mu)}(X_1)_{(x,\mu)}&=\bigg(A_x,\frac{1}{2\pi i}\frac{dz}{z}\bigg|_\mu\bigg)(0_x, 2\pi i \mu)\\
&=B_x(0_x)+\frac{1}{2\pi i}\frac{dz}{z}\bigg|_\mu(2\pi i \mu)\\
&=\frac{1}{2\pi i}\frac{2\pi i \mu}{\mu}\\
&=1.
\end{align*}
Hence, $\varphi(X_C)=C\varphi(X_1)=C$.

Next, to show the equivariance, let $(v_x,\eta)\in T_xM\times \bC$ and $\lambda\in\bC^*$, then
\begin{align*}
(\lambda^*\varphi)_{(x,\mu)}(v_x,\eta)&=\varphi_{(x,\lambda\mu)}(v_x,\lambda\eta)\\
&=B_x(v_x)+\frac{1}{2\pi i}\frac{dz}{z}\bigg|_{\lambda\mu}(\lambda\eta)\\
&=B_x(v_x)+\frac{1}{2\pi i}\frac{\eta}{\mu}\\
&=\varphi_{(x,\mu)}(v_x,\eta).
\end{align*}
Thus, $\varphi$ is a connection 1-form.
\end{egs}

Given a connection 1-form, as was hinted above the defintion, one can find a way of splitting the tangent bundle in a way which is transverse to the vertical bundle $VP$. This splitting due to the connection 1-form is called the Horizontal bundle. More formally.

\begin{defs}
Given a $\bC^*$-principal bundle $P\ra M$,  a connection 1-form $\varphi$, and $p\in P$, we say $v\in T_pP$ is horizontal if $\varphi(v)=0$. Define the horizontal bundle $HP$ by
\bge
HP:=\{v\in TP \ | \ v\text{ is horizontal}\}.
\ene
\end{defs}

\begin{prop}
$HP$ is a smooth subbundle of $TP$. Furthermore, there is a splitting
\bge
TP=VP\oplus HP.
\ene
\end{prop}

\begin{proof}
For each $p\in P$, $H_pP=\ker \varphi_p$. Now, since the map
\bge
\bC\ra \Gamma(VP);\quad C\mapsto X_C
\ene
is a Lie algebra isomorphism, this implies $\varphi_p|_{V_pP}:V_pP\ra \bC$ is the inverse and hence $\dim H_pP=\dim P-1$. Taking a trivializing neighbourhood in $M$ for $P$, it suffices to show that the Horizontal bundle is smooth on the trivial bundle. See \cite{greub} for a proof in this case.

It's a fact from linear algebra that each tangent space $T_pP=V_pP\oplus H_pP$. And hence $TP=VP\oplus HP$.
\end{proof}

\begin{cor}
The projection maps
\bge
Ver:TP\ra VP,\quad Hor:TP\ra HP
\ene
are smooth.
\end{cor}

Note that by dimensionality arguments alone, $d\sigma_p:H_pP\ra T_{\sigma(p)}M$ is an isomorphism of vector spaces for all $p\in P$. This allows us to lift vector fields on $M$ to vector fields on $P$.

\begin{thm}
Let $\xi\in\mfX(U)$ be a local section of $M$. Then there exists a unique local vector field $\overline{\xi}\in \mfX(\pi^{-1}(U))$ such that 
\begin{itemize}
    \item[(i)] $Hor(\overline{\xi})=\overline{\xi}$ 
    \item[(ii)] $d\pi_p(\overline{\xi}_p)=\xi_{\pi(p)}$
    \item[(iii)] For all $\lambda\in\bC^*$ we have
    \bge
    \lambda_*\overline{\xi}=\overline{\xi}.
    \ene
\end{itemize}
\end{thm}

\begin{proof}
See \cite{koba}
\end{proof}

\begin{defs}
If $\xi\in\mfX(U)$ and $\overline{\xi}\in\mfX(\sigma^{-1}(U))$ is as in the above theorem, then we call $\overline{\xi}$ the horizontal lift of $\xi$.
\end{defs}

\begin{cor}
For each open subset $U\subset M$, the map
\bge
\mfX(U)\ra \mfX(\pi^{-1}(U));\quad \xi\mapsto \overline{\xi}
\ene
is linear.
\end{cor}

\subsection{Covariant Derivatives}
The main result of this subsection is the correspondence between covariant derivatives and connection 1-forms. We will need $M$ to be paracompact for all these proofs to work. The layout of this subsection closely follows \cite{jean}. 

Throughout this section, let $L\xrightarrow{\pi}M$ be a line bundle over $M$ and $L^+\xrightarrow{\widetilde{\pi}}M$ be its frame bundle.

Let's start with an example of a covariant derivative.

\begin{egs}
Let $L=M\times\bC$ be the trivial bundle over $M$. If $U\subset M$ is an open subset, then for all local sections $s\in\Gamma(U,L)$, there exists unique $f_s\in C^\infty(U,\bC)$ such that
\bge
s(x)=(x,f_s(x)),\quad x\in M.
\ene
For every $\xi\in\mfX_\bC(U)$, define
\bge
\nabla_\xi s(x):=(x,\xi_xf)
\ene
Since tangent vectors are derivations, we immediately obtain that $\nabla$ is a covariant derivative.
\end{egs}

Before we can continue, we need some ideas from point-set topology.

\begin{defs}
Let $X$ be a topological space.
\begin{itemize}
    \item[(i)] An open cover $\{U_i\}_{i\in I}$ is said to be locally finite if for all $x\in M$, there exists a neighbourhood $U$ of $x$ such that $U\cap U_i\neq\emptyset$ for only finitely many $i$.
    \item[(ii)] $X$ is said to be paracompact if for every open cover $\{U_i\}_{i\in I}$ there exists a refinement $\{V_j\}_{j\in J}$ which is a locally finite open cover of $X$.
\end{itemize}
\end{defs}

\begin{prop}
If $M$ is paracompact, then any open cover admits a partition of unity.
\end{prop}

\begin{proof}
See Lee \cite{lee}.
\end{proof}

We shall from now on assume that $M$ is paracompact. Note that since all smooth manifolds are paracompact, this will have no impact on the statement of the main theorem. 

Using the paracompactness of $M$, we may now show a foundational result.

\begin{prop}
If $Co(L)$ denotes the set of all covariant derivatives on $L\ra M$, then $Co(L)$ is nonempty.
\end{prop}

\begin{proof}
Let $\{U_i\}_{i\in I}$ be a trivializing cover of $L\ra M$. Since $\pi^{-1}(U_i)\ra U_i$ is isomorphic to the trivial bundle for all $i$, we have there exists a covariant derivative $\nabla_i$ on $\pi^{-1}(U)\ra U_i$ for all $i$.

Take a partition of unity $\{h_i\}_{i\in I}$ subordinate to the cover. Define
\bge
\nabla:=\sum_{i\in I} h_i\nabla_i.
\ene
That is, if $U\subset M$ is open, $\xi\in\mfX(U)$ is a local vector field, and $s\in\Gamma(U,L)$ is a local section of $L$, then define
\bge
\nabla_\xi s:=\sum_{i\in I} h_i(\nabla_i)_{\xi_i}s_i
\ene
where $\xi_i=\xi|_{U_i}$ and $s_i=s|_{U_i}$. 
\end{proof}

Now, if $\nabla\in Co(L)$ and $\alpha\in\Omega^1(M)\otimes\bC$, then we can define a new covariant derivative by
\bge
(\nabla+\alpha)_\xi s=\nabla_\xi s+\alpha(\xi) s.
\ene
Turns out, this proceedure generates all covariant derivatives.

\begin{cor}
$Co(L)$ is an affine space under the above action by the vector space $\Omega^1(M)\otimes\bC$.
\end{cor}

\begin{rem}
For notational convenience, if $U\subset M$ is open, $s\in\Gamma(U,L^+)$ and $t\in\Gamma(U,L)$, then define the smooth function
\bge
\frac{t}{s}:U\ra\bC
\ene
uniquely by the property
\bge
t=\frac{t}{s}s.
\ene
Note that if $t$ also doesn't vanish (i.e. $t\in\Gamma(U,L^+)$), then $t/s$ doesn't vanish as well.
\end{rem}

\begin{proof}(Of Corollary)

Let $\nabla,\nabla'\in Co(L)$. For any open subset $U\subset M$,  $s\in\Gamma(U,L^+)$, and $\xi\in\mfX(U)$, define
\bge
A_U(\xi):=\frac{\nabla_\xi s-\nabla'\xi s}{s}.
\ene
$A_U$ is linear in $\xi$ and smooth. Hence $A_U\in\Omega^1(U)\otimes\bC$. 

Choosing a different $t\in\Gamma(U,L^+)$, we obtain $g:U\ra \bC^*$ such that $t=gs$. Then we compute using the fact that both $\nabla$ and $\nabla'$ are derivations in the second entry:
\bge
\frac{\nabla_\xi t-\nabla'\xi t}{t}=\frac{\nabla_\xi gs-\nabla'\xi gs}{gs}=\frac{\nabla_\xi s-\nabla'\xi s}{s}.
\ene
Thus, $A_U$ doesn't depend on choice of non-vanishing section. It's then clear that if $U\cap V\neq\emptyset$, then $A_U|_{U\cap V}=A_V|_{U\cap V}$. Hence the $A_U$ form a globally defined 1-form $A\in\Omega^1(M)\otimes\bC$. By construction, we have
\bge
\nabla-\nabla'=A.
\ene
$A$ is clearly the unique 1-form to have this property.
\end{proof}

We will show a correspondence between covariant derivatives and connection 1-forms. First, we need a definition and a Lemma.

\begin{defs}
Let $V\subset L^+$ be an open subset and $f\in C^\infty(V,\bC)$. Say $f$ is $\bC^*$-equivariant if for all $z\in V$ and $\lambda\in\bC^*$
\bge
f(\lambda z)=\lambda^{-1}f(z).
\ene
Denote the set of all such equivariant funtions on $V$ by $\bC^*(V)$.
\end{defs}

The following correspondence is from  \cite{Snia}.
\begin{lem}
Let $U\subset M$ be open.
\begin{itemize}
    \item[(i)] If $s\in\Gamma(U,L)$, then there exists a unique $\phi_s\in \bC^*(\widetilde{\pi}^{-1}(U))$ such that for all $z\in \widetilde{\pi}^{-1}(U)$
    \bge
    s(\pi(z))=\phi_s(z)z.
    \ene
    (Recall that $L^+\subset L$ so $\phi_s(z)z$ is defined as an element of $L$ for any value of $\phi_s$).
    \item[(ii)] If $f\in \bC^*(\widetilde{\pi}^{-1}(U))$, then there exists a unique $\rho_f\in \Gamma(U,L)$ such that for all $z\in \widetilde{\pi}^{-1}(U)$
    \bge
    \rho_f(\pi(z))=f(z)z.
    \ene
    \item[(iii)] The maps
    \begin{align*}
    \phi:\Gamma(U,L)\ra \bC^*(\widetilde{\pi}^{-1}(U));&\quad s\mapsto \phi_s\\
    \rho:\bC^*(\widetilde{\pi}^{-1}(U))\ra\Gamma(U,L);&\quad f\mapsto \rho_f
    \end{align*}
    are linear and inverse to one another.
\end{itemize}
\end{lem}

\begin{proof}
Fix $s\in\Gamma(U,L)$ and $f\in\bC^*(\widetilde{\pi}^{-1}(U))$.

It's clear that $\phi_s$ is well-defined. So we only show that $\rho_f$ is well-defined.

Suppose $\lambda\in \bC^*$ and $z\in\widetilde{\pi}^{-1}(U)$. Then
\bge
f(\lambda z)( \lambda z)=\lambda^{-1}\lambda f(z)z=f(z)z.
\ene
Hence, $\rho_f$ is well-defined.

We'll show $\phi_s$ is smooth. It's the exact same computation to show that $\rho_f$ is smooth.

Let $U\subset M$ be a trivializing neighbourhood and $\psi:\pi^{-1}(U)\ra U\times\bC$ a trivializing diffeomorphism. By definition, there exists $\beta:\pi^{-1}(U)\ra\bC$, which is a linear isomorphism over each fibre, such that
\bge
\psi(z)=(\pi(z),\beta(z)).
\ene
Note that for any $z\in\widetilde{\pi}^{-1}(U)\subset\pi^{-1}(U)$ we have $\beta(z)\neq 0$. 

Fix $z\in\widetilde{\pi}^{-1}(U)$. We have
\bge
\psi\circ s(\pi(z))=(\pi(z), F(\pi(z)))
\ene
for some $F:U\ra\bC$. On the otherhand, we have 
\bge
\psi\circ s(\pi(z))=\psi(\phi_s(z)z)=(\pi(z),\phi_s(z)\beta(z)).
\ene
Hence, we have
\bge
\phi_s(z)=\frac{f(\pi(z))}{\beta(z)}.
\ene
Both $F$ and $\beta$ are smooth and $\beta\neq 0$ on $\widetilde{\pi}^{-1}(U)$, hence $\phi_s$ is smooth.

It's clear that both $\phi$ and $\rho$ are linear. Furthermore, their defining equations are inverse to one another. Hence, $\phi$ and $\rho$ are inverse.
\end{proof}

Now, we are ready for the main result of this section. The proof of (i) is due to Śniatycki \cite{Snia} and the proof of (ii) is the argument given by  Brylinski \cite{jean}.

\begin{thm}
\begin{itemize}
    \item[(i)] If $\varphi\in\Omega^1(L^+)$ is a connection 1-form, then there exists a unique covariant derivative $\nabla$ such that
    \begin{equation}
    \nabla_\xi s=2\pi i(s^*\varphi(\xi))s.
    \end{equation}
    \item[(ii)] If $\nabla$ is a covariant derivative, then there exists a unique connection 1-form $\varphi\in\Omega^1(L^+)$ satisfying equation (1).
\end{itemize}
\end{thm}

\begin{proof}
\begin{itemize}
    \item[(i)] Let $\varphi\in\Omega^1(L^+)\otimes\bC$ be a connection 1-form.
    
    Let $U\subset M$ be an open subset, $\xi\in\mfX_\bC(U)$ a local vector field, and $s\in\Gamma(U,L)$ a local section of $L$. If $\overline{\xi}\in\mfX_\bC(\pi^{-1}(U))$ is the horizontal lift of $\xi$ and $\phi_s\in C^\infty(\widetilde{\pi}^{-1}(U),\bC)$ is the unique $\bC^*$-equivariant smooth function such that
    \bge
    s(\pi(z))=\phi_s(z)z,\quad z\in \widetilde{\pi}^{-1}(U)
    \ene
    then define
    \bge
    \nabla_\xi s(\pi(z)):=(\overline{\xi}_z\phi_s)z.
    \ene
    
    If $\nabla_\xi s$ is well-defined, then since it's produced from the composition of smooth operations, $\nabla_\xi s\in\Gamma(U,L)$.
    
    To show well-definedness, let $\lambda\in\bC^*$, then 
    \begin{align*}
    \nabla_\xi s(\pi(\lambda z))&=(\overline{\xi}_{\lambda z}\phi_s)\lambda z\\
    &=((\lambda_*\overline{\xi}_z)\phi_s)\lambda z\\
    &=(\overline{\xi}_z(\phi_s\circ\lambda))\lambda z\\
    &=(\overline{\xi}_z\phi_s)z\\
    &=\nabla_\xi s(\pi(z)).
    \end{align*}
    Hence, $\nabla_\xi s$ is well-defined and smooth.
    
    To see that $\nabla$ is a covariant derivative, first note that both the operations
    \begin{align*}
    \mfX_\bC(U)\ra \mfX_\bC(\widetilde{\pi}^{-1}(U));&\quad \xi\mapsto \overline{\xi}\\
    \Gamma(U,L)\ra C^\infty(\widetilde{\pi}^{-1}(U),\bC);&\quad s\mapsto \phi_s
    \end{align*}
    are linear. Hence, $\nabla$ is linear in each entry. Second, since both the above maps respect restriction, so does $\nabla$.
    
    To finish showing that $\nabla$ is a covariant derivative, we need to show that it is a derivation in the second entry. For this purpose, let $f\in C^\infty(U,\bC)$ and $z\in\widetilde{\pi}^{-1}(U)$. Then we have
    \begin{align*}
    \nabla_\xi(fs)(\pi(z))&=(\overline{\xi}_z \phi_{fs})z\\
    &=(\overline{\xi}_z f\circ\pi \phi_{s})z\\
    &=((\overline{\xi}_z f\circ\pi) \phi_{s}(z)+f(\pi(z))\overline{\xi}_z \phi_s)z\\
    &=(\xi_{\pi(z)}f)s(\pi(z))+f(\pi(z))\nabla_\xi (\pi(z)).
    \end{align*}
    Since $z$ was arbitrary and $\widetilde{\pi}:L^+\ra M$ is surjective, the result follows.
    
    Finally, we need to demonstrate equation (1) holds.
    Let $s\in\Gamma(U,L^+)$ and $\xi\in \mfX_\bC(M)$. Observe that for all $x\in M$, we have $ds_x(\xi_x)\in T_{s(x)}L^+$. Hence, we can decompose
    \bge
    ds_x(\xi_x)=Hor(ds_x(\xi_x))+Ver(ds_x(\xi_x)).
    \ene
    Now, observe that since $\pi\circ s=\id$, we have 
    \bge
    Hor(ds_x(\xi_x))=\overline{\xi}_{s(x)}.
    \ene
    Recall the fundamental vector field $X_1$ on $L^+$ as in section 3.1. It's easy to see that
    \bge
    Ver(ds_x(\xi_x))=\varphi(ds_x(\xi_x))(X_1)_{s(x)}.
    \ene
    Thus,
    \bge
    ds_x(\xi_x)=\overline{\xi}_{s(x)}+\varphi(ds_x(\xi_x))(X_1)_{s(x)}.
    \ene
    Using this, we may then compute 
    \begin{align*}
    \overline{\xi}_{s(x)} \phi_s&=(dF_s)_{s(x)}(\overline{\xi}_{s(x)})\\
    &=(d\phi_s)_{s(x)}(ds_x(\xi_x)-\varphi(ds_x(\xi_x))(X_1)_{s(x)})\\
    &=d(\phi_s\circ s)_x(\xi_x)-\varphi(ds_x(\xi_x))(d\phi_s)_{s(x)}((X_1)_{s(x)})\\
    &=2\pi i s^*\varphi(\xi_x).
    \end{align*}
    Hence, since $\overline{\xi}$ is invariant, one can compute for any $z\in L^+$
    \bge
    \nabla_\xi s(\pi(z))=(\overline{\xi}_z \phi_s)z=[2\pi i s^*\varphi(\xi_{\pi(z)}) \phi_s(z)]z=[2\pi i s^*\varphi(\xi_{\pi(z)})]s(\pi(z)).
    \ene
    
    \item[(ii)] Let $\nabla$ be a covariant derivative on $L$. If the connection 1-form $\varphi$ exists, then it is uniquely characterized by equation (1). For if $\varphi'$ is another such connection 1-form, then we have for all local sections $s$ on $L^+$
    \bge
    s^*\varphi=s^*\varphi'.
    \ene
    Since every element of the tangent space of $L^+$ can be obtained from the image of a local section, we thus conclude $\varphi=\varphi'$. Hence, it suffices to show that $\varphi$ exists locally, i.e. on the trivial bundle.
    
    Let $L=M\times \bC$ be the trivial bundle. As in a previous example, we may identify the sections of $L$ with complex valued functions on $M$, i.e. a section of $L$ is of the form $(\id_M,f)$, where $f\in C^\infty(M,\bC)$. Furthermore, we showed that
    \bge
    \mfX_\bC(M)\times \Gamma(L)\ra \Gamma(L);\quad (\xi,(\id_M,f))\mapsto (\id, df(\xi))
    \ene
    is a covariant derivative. Hence, since $Co(L)$ is affine, if $\nabla$ is a covariant derivative on $L$, there exists $B\in\Omega^1(M)\otimes\bC$ such that for all $f\in C^\infty(M,\bC)$
    \bge
    \nabla_\xi s=(\id_M,df(\xi)+2\pi if B(\xi)).
    \ene
    If $z$ is the coordinate on $\bC$, define 
    \bge
    \varphi=\bigg(B,\frac{1}{2\pi i}\frac{dz}{z}\bigg)\in\Omega^1(L)\otimes\bC.
    \ene 
    But we showed in an even earlier example that $\varphi$ is a connection 1-form. So all that's left to do is to verify equation (2).
    
    We then see that if $s(x)=(x,f(x))$, then
    \begin{align*}
    2\pi i s^*\varphi(\xi)s&=2\pi i \varphi(\xi, df(\xi))s\\
    &=2\pi i\bigg(B,\frac{1}{2\pi i}\frac{dz}{z}\bigg(\xi,df(\xi)))s\\
    &=2\pi i\bigg[B(\xi)+\frac{1}{2\pi i}\frac{df}{f}(\xi)\bigg]s\\
    &=(\id, df(\xi)+2\pi i f B(\xi)).
    \end{align*}
    And this concludes the proof.
\end{itemize}
\end{proof}

We can use this theorem to get a description of the curvature of the covariant derivative. Fix a line bundle $(L,\nabla)$ with covariant derivative over $M$ and let $\varphi$ the corresponding connection 1-form on $L^+$.

\begin{cor}
If $R$ is the curvature of $\nabla$, then for any open subset $U\subset M$ and non-vanishing section $s\in\Gamma(U,L^+)$, we have
\bge
2\pi i s^*d\varphi=R.
\ene
\end{cor}

\begin{proof}
Set $\alpha=s^*\varphi$. We then have for any local vector field $\xi$ on $M$
\bge
\nabla_\xi s=2\pi i \alpha(\xi)s.
\ene
Now, if $\eta$ is any other local vector field, we can compute
\begin{align*}
\nabla_\xi(\nabla_\eta s)&=2\pi i\nabla_\xi(\alpha(\eta) s)\\
&=2\pi i[\xi(\alpha(\eta))s+\alpha(\eta)\nabla_\xi s]\\
&=2\pi i \xi(\alpha(\eta))s-4\pi^2\alpha(\eta)\alpha(\xi)s.
\end{align*}
Using this, we can then compute,
\bge
R(\xi,\eta)s=[\nabla_\xi,\nabla_\eta]s-\nabla_{[\xi,\eta]}s=2\pi i d\alpha(\xi,\eta)s.
\ene
Hence,
\bge
2\pi id\alpha=2\pi id(s^*\varphi)=2\pi is^*d\varphi=R.
\ene
\end{proof}

From here, one can deduce the following corollary.

\begin{cor}
The curvature $R$ of $\nabla$ is a closed complex 2-form on $M$ satisfying.
\begin{equation}
\pi^*R=2\pi id\varphi.
\end{equation}
\end{cor}

\begin{proof}
It's clear by the above corollary that $R\in\Omega^2(M)\otimes \bC$. It's also clear that $R$ is closed since
\bge
dR=2\pi i d(s^*d\varphi)=2\pi is^*d^2\varphi=0.
\ene
To show equation (3), let $U\subset M$ be an open subset and $s\in\Gamma(U,L^+)$. Note that since $\pi\circ s=\id_U$, we have
\bge
R=s^*\pi^*R=2\pi is^*d\varphi.
\ene
Hence, $\pi^* R-2\pi id\varphi\in\ker s^*$. Since $s$ was arbitrary, we conclude the corollary.
\end{proof}

\subsection{Local Systems on a Line Bundle}

We finish this section with a technical discussion on the local nature of line bundles. This will be useful when we want to construct a line bundle from a closed 2-form on a manifold. This discussion closely follows \cite{kostant} and \cite{jean}.

\begin{defs}
A good cover of $M$ is an open cover $\{U_i\}_{i\in I}$ such that for all $i_0,\dots,i_m\in I$ we have $U_{i_0}\cap\cdots\cap U_{i_m}$ is either empty or connected and simply connected.
\end{defs}

\begin{prop}
On a paracompact manifold, every open cover admits a refinement which is a good cover.
\end{prop}

\begin{proof}
See \cite{jean}
\end{proof}

Let $M$ be a paracompact manifold, $L\xrightarrow{\pi}M$ a line bundle over $M$, and $L^+\xrightarrow{\widetilde{\pi}}M$ denote the frame bundle of $L$ over $M$, where
\bge
\widetilde{\pi}=\pi|_{L^+}.
\ene

\begin{defs}
A local system of $L\ra M$ is a collection of pairs $\{(U_i,s_i)\}_{i\in I}$ such that
\begin{itemize}
    \item[(i)] $\{U_i\}_{i\in I}$ is a good cover of $M$
    \item[(ii)] For all $i\in I$, $s_i\in\Gamma(U_i,L^+)$.
\end{itemize}
If $U_i\cap U_j\neq\emptyset$, then define $g_{i}:U_i\cap U_j\ra\bC^*$ uniquely by
\bge
s_i=g_{ij} s_j.
\ene
Call the $g_{ij}$ the transition functions of the local system.
\end{defs}

In a local system, we can break our analysis of a line bundle down into the discussion of the trivial bundle, locally. We could always do this of course, but a local system then gives us a way of piecing the local analysis together again to give a global result. That's great, but we first need them to exist.

\begin{prop}
$L\ra M$ admits a local system.
\end{prop}

\begin{proof}
Let $\{U_i\}_{i\in I}$ be an open cover of $M$ such that for all $i\in I$ there exists a trivialization 
\bge
\psi_i:\pi^{-1}(U_i)\ra U_i\times\bC.
\ene
For all $x\in U_i$, define
\bge
s_i(x):=\psi_i^{-1}(x,1).
\ene
Since $\psi_i$ is a diffeomorphism and a linear isomorphism on each fibre, we have $s_i\in\Gamma(U_i,L^+)$.

Now, since $M$ is paracompact, we may refine the cover $\{U_i\}_{i\in I}$ into a good cover. Since refinement replaces our sets with even smaller sets, we will still have a trivializing cover. Hence, after some refinement, we may make $\{(U_i,s_i)\}_{i\in I}$ into a local system.
\end{proof}

A useful, but easy result is the following fact. We shall make use of it many times throughout this paper.
\begin{fact}
If $\{(U_i,s_i)\}_{i\in I}$ is a local system and $g_{ij}$ are the transition functions, then if $U_i\cap U_j\cap U_k\neq\emptyset$, we have
\begin{equation}
g_{ik}=g_{ij}g_{jk}.
\end{equation}
We call equation (4) the cocycle condition.
\end{fact}

Just as we can piece together a manifold from local data, we can also piece a line bundle back together from a collection of transition functions. In a sense there is a one-to-one correspondence between transition functions and line bundles, see \cite{jean} for a more precise formulation of this result.

\begin{prop}
Let $\{U_i\}_{i\in I}$ be a good cover of $M$ and suppose for each $i,j\in I$ such that $U_i\cap U_j\neq\emptyset$, there exists smooth $g_{ij}:U_i\cap U_j\ra\bC^*$ satisfying the coycle condition of equation (4). Then there exists a line bundle $L\ra M$ with local system $\{(U_i,s_i)\}_{i\in I}$ such that $g_{ij}$ are the transition functions.
\end{prop}

\begin{proof}
We will use the good cover $\{U_i\}_{i\in I}$ and the $g_{ij}$ to construct a $\bC^*$-principal bundle, then use that to produce a line bundle which matches our requirements.

Let
\bge
W=\bigsqcup_{i\in I} U_i\times\bC^*.
\ene
Give $W$ the trivial action by $\bC^*$: if $(x,\mu)\in U_i\times\bC^*$ and $\lambda\in\bC^*$, define
\bge
(x,\mu)\cdot\lambda:=(x,\mu\lambda).
\ene

We impose equivalence relation on $W$: if $(x,\mu)\in U_i\cap U_j\times\bC^*$ then
\bge
(x,\mu)\sim (x,g_{ij}(x)\mu).
\ene
Define $P=W/\sim$ and let $Q:W\ra P$ be the natural quotient map. 

It's easy to see that the action of $\bC^*$ descends to $P$ and that this action is free. Further, if we restrict the quotient map $Q:W\ra P$ to $Q_i:U_i\times\bC^*\ra P$, then this map is a homeomorphism. Hence, we may use the smooth structure on $W$ to give $P$ a smooth structure. 

Finally, if we equip $P$ with the projection
\bge
\sigma:P\ra M;\quad Q_i(x,\mu)\mapsto x,\quad x\in U_i.
\ene
then $P$ is a $\bC^*$-principal bundle. See \cite{jean} for a more complete discussion of this construction.

Note that by construction, if $x\in U_i\cap U_j$, then
\bge
Q_j(x,\mu)\cdot g_{ij}(x)=Q_j(x,\mu g_{ij}(x))=Q_i(x,\mu).
\ene

Now, let's construct our line bundle. Let
\bge
L=(P\times\bC)/\sim
\ene
where $(p,z)\sim (p\cdot\lambda,\lambda^{-1}z)$. Give it projection 
\bge
\pi([p,z]):=\sigma(p).
\ene
It's an easy exercise to show that $\pi$ is well-defined. Further, if we fix $x\in M$, we can give $L_x$ a vector space structure as follows. If $[p,z],[q,w]\in \pi^{-1}(x)$, then there exists $\lambda\in\bC^*$ such that $q=p\cdot \lambda$. Define
\bge
[p,z]+[q,w]:=[p,z+\lambda w].
\ene
Again, another easy exercise to show that this addition is well-defined. For scalar multiplication, simply define
\bge
\lambda[p,z]:=[p,\lambda z],\quad\lambda\in\bC.
\ene

To show that $L$ is a smooth manifold and that $L\xrightarrow{\pi}M$ is locally trivial amounts to showing that the local trivialization which defines $P$ induces a local trivialization on $L$. See \cite{greub} for a for a more complete discussion of obtaining a line bundle from a $\bC^*$-principal bundle.

Finally, we need our local sections. For each $i\in I$, define
\bge
s_i(x)=[Q_i(x,1),1]\in L.
\ene
By construction, $s_i$ is a smooth non-vanishing section for all $i$. To see that that $g_{ij}$ are the transition functions, let $x\in U_i\cap U_j$. Then
\bge
s_j(x)\cdot g_{ij}(x)=[Q_j(x,1),g_{ij}(x)]=[Q_j(x,1)g_{ij}(x),1]=[Q_i(x,1),1]=s_i(x).
\ene
This completes the proof.
\end{proof}

\begin{lem}
Let $M$ be a manifold and $L=M\times\bC$ the trivial bundle. If $s\in\Gamma(M,\bC^*)$ and $\alpha\in\Omega^1(M)\otimes\bC$, then there exists unique covariant derivative $\nabla$ on $L$ such that for all local vector fields $\xi$ on $M$
\bge
\alpha(\xi)=\frac{\nabla_\xi s}{2\pi i s}.
\ene
\end{lem}

\begin{proof}
Let $r=(\id_M,f_r)$ be a section of $L$. We've shown already in a previous example that
\bge
\nabla'_\xi r(x)=(x,\xi_x f_r)
\ene
is a covariant derivative.

Thus, for our non-vanishing section $s\in \Gamma(M,L^+)$, written $s=(\id_M,f)$, if we define the complex 1-form
\bge
A=2\pi i\alpha-\frac{df}{f},
\ene
then $\nabla=\nabla'+A$ is also a covariant derivative. We then simply compute for a local vector field $\xi$ on $M$:
\bge
\nabla_\xi s(x)=\bigg(x,\xi_xf_s+f(x)\bigg[2\pi i\alpha_x(\xi_x)-\frac{\xi_x f}{f}\bigg]\bigg)=(x, 2\pi i f(x)\alpha_x(\xi_x)).
\ene
Hence,
\bge
\nabla_\xi s=2\pi i\alpha(\xi)s
\ene
and we're done on existence. Uniqueness easily follows.
\end{proof}

\begin{thm}
Let $L\xrightarrow{\pi} M$ be a line bundle, $\{(U_i,s_i)\}_{i\in I}$ a local system with transition functions $g_{ij}$, and $\alpha_i\in\Omega^1(U_i)\otimes\bC$ a collection of complex 1-forms so that
\bge
\alpha_i-\alpha_j=\frac{1}{2\pi i}\frac{d g_{ij}}{g_{ij}}.
\ene
Then there exists a unique covariant derivative $\nabla$ such that if $\varphi$ is the corresponding connection 1-form and $R$ the curvature, we have
\begin{itemize}
    \item[(i)] $\alpha_i=s_i^*\varphi$
    \item[(ii)] $2\pi i d\alpha_i=R|_{U_i}$.
\end{itemize}
\end{thm}

\begin{proof}
On $U_i$, by the previous proposition, there exists a unique covariant derivative $\nabla_i$ on $U_i$ such that
\bge
\alpha_i=\frac{\nabla_i s_i}{2\pi is_i}.
\ene
We will show that $\nabla_i=\nabla_j$ on $U_i\cap U_j$, which will then imply they piece together to form a globally defined covariant derivative. Suffices to show that
\bge
\nabla_j s_i=\nabla_i s_i.
\ene

Fix a local vector field $\xi$. Then we have
\begin{align*}
(\nabla_j)_\xi s_i&=(\nabla_j)_\xi (g_{ij}s_j)\\
&=(\xi g_{ij})s_j+g_{ij}(\nabla_j)_\xi s_j\\
&=dg_{ij}(\xi) s_j+2\pi ig_{ij}\alpha_j(\xi) s_j\\
&=2\pi i\bigg(\frac{1}{2\pi i}\frac{d g_{ij}(\xi)}{g_{ij}}+\alpha_j(\xi)\bigg)s_i\\
&=2\pi i\alpha_i(\xi)s_i\\
&=(\nabla_i)_\xi s_i.
\end{align*}

Now, let $\varphi$ be the corresponding connection 1-form. Then we have by previous theorem
\bge
\nabla_\xi s=2\pi is^*\varphi(\xi)s
\ene
for every nonvanishing local section $s$ of $L$. But this implies on $U_i$
\bge
\alpha_i=\frac{\nabla s_i}{2\pi is_i}=s^*\varphi.
\ene

It's then a simple matter of applying Corollary 5 to obtain
\bge
R|_{U_i}=2\pi i s_i^*d\varphi=2\pi i d\alpha_i.
\ene
This completes the proof.
\end{proof}

\newpage

\section{Existence of a Prequantization}
Before we prove the main theorem, there is the pesky matter of the definition of an integral cohomology class. We will need some \v{C}ech cohomology and a characterization of integrality given by Kostant \cite{kostant}. This tangent will be supplied by \cite{jean} and \cite{bott}.

\subsection{A Little \v{C}ech Cohomology}
The main uses of \v{C}ech cohomology for us will be the fact that we can compute cohomology with coefficents in any ring, so in particular $\bZ$. And, as we shall see, \v{C}ech cohomology is built into line bundles via local systems. 

As with any cohomology theory, we need cochains.

\begin{defs}
Let $\msU=\{U_i\}_{i\in I}$ be a cover of a topological space $X$ and $R$ a ring. For all $i_0,\dots,i_p\in I$, write
\bge
U_{i_0\cdots i_p}:=U_{i_0}\cap\cdots\cap U_{i_p}.
\ene
A Cech $p$-cochain $\mu$ with respect to $\msU$ is a collection of constant functions $\mu_{i_0\cdots i_p}:U_{i_0\cdots i_p}\ra R$ for all $i_0,\dots,i_p\in I$ with $U_{i_0\cdots i_p}\neq\emptyset$. Let $\check{C}^p(\msU,R)$ denote the collection of all \v{C}ech $p$-cochains.
\end{defs}

\begin{fact}
$\check{C}^p(\msU,R)$ is an abelian group for all $p$.
\end{fact}

To get a cohomology theory, we now need a differential.

\begin{defs}
Let $\mu\in \check{C}^p(\msU,R)$. Define
\bge
(\delta\mu)_{i_0\cdots i_{p+1}}:=\sum_{j=0}^{p+1}(-1)^j\mu_{i_0\cdots i_{j-1}i_{j+1}\cdots i_p}.
\ene
\end{defs}

In order $\delta$ to make sense as a differential, we need it to be an abelian group homomorphism and it needs to square to 0. Thankfully, as the following fact shows, this is the case.

\begin{fact}
\begin{itemize}
    \item[(i)] $\delta:\check{C}^p(\msU,R)\ra \check{C}^{p+1}(\msU,R)$ is an abelian group homomorphism.
    \item[(ii)] $\delta^2=0$.
\end{itemize}
\end{fact}

\begin{defs}
Define the \v{C}ech $p$-cocycles and $p$-coboundaries, respectively, by
\begin{align*}
\check{Z}^p(\msU,R)&:=\{\mu\in \check{C}^p(\msU,R) \ | \ \delta\mu=0\}.\\
\check{B}^p(\msU,R)&:=\{\mu\in \check{C}^p(\msU,R) \ | \ \exists \ \eta\in C^{p-1}(\msU,R) \ : \ \mu=\delta\eta\}.
\end{align*}
Define the $p$-th \v{C}ech cohomology group of $X$ with respect to $\msU$ with coefficients in $R$ by
\bge
\check{H}^p(\msU,R)=\check{Z}^p(\msU,R)/\check{B}^p(\msU,R).
\ene
\end{defs}

It might seem like the cohomology will depend on the choice of an open cover. Thankfully, in the case of a smooth manifold and coefficients in $\bR$, this is not the case.

\begin{thm}
If $M$ is a manifold and $\msU$ is a good cover of $M$, then for all $p$
\bge
\check{H}^p(\msU;\bR)\cong H^p(M).
\ene
\end{thm}

\begin{proof}
See Bott and Tu \cite{bott}
\end{proof}

To finish off this section, we define what it means that a closed 2-form is integral.

\begin{fact}
Let $\msU$ be an open cover of $X$, let $R$ and $S$ two rings, and $\varphi:R\ra S$ an abelian group homorphism. Then there is an induced group homomorphism
\bge 
\check{H}^p(\msU,R)\ra \check{H}^p(\msU,S)
\ene
for all $p$.
\end{fact}

\begin{proof}
Fix $p$ and let $\mu\in \check{C}^p(\msU,R)$. Define $\varphi(\mu)$ by
\bge
\varphi(\mu)_{i_0\cdots i_p}:U_{i_0\cdots i_p}\ra S;\quad x\mapsto \varphi(\mu_{i_0\cdots i_p}(x)).
\ene
This clearly induces a map of abelian groups
\bge
\varphi:\check{C}^p(\msU,R)\ra \check{C}^p(\msU,S).
\ene

To show that this induces a homomorphism at the level of cohomology, we show $\delta\varphi=\varphi\delta$. To do this, let $\mu\in \check{C}^p(\msU,R)$. Then we have
\begin{align*}
(\varphi \delta \mu)_{i_0\cdots i_{p+1}}&=\varphi\bigg(\sum_{j=0}^{p+1}(-1)^j\mu_{i_0\cdots i_{j-1}i_{j+1}\cdots i_p}\bigg)\\
&=\sum_{j=0}^{p+1}(-1)^j\varphi(\mu)_{i_0\cdots i_{j-1}i_{j+1}\cdots i_p}\\
&=\delta \varphi(\mu).
\end{align*}
\end{proof}

\begin{defs}
Let $\msU$ be a good cover of the manifold $M$ and $\gamma_\msU:\check{H}^2(\msU;\bZ)\ra H^2(M)$ be the composition of the maps
\bge
\check{H}^2(\msU;\bZ)\ra \check{H}^2(\msU;\bR)\ra \check{H}^2(M).
\ene
Then we say $[\omega]\in H^2(M)$ is integral if it lies the in the image of $\gamma_\msU$.
\end{defs}

This definition depended on the good cover chosen. However, as a corollary to the above theorem, we can see that it does not.

\begin{cor}
If $\msU$ and $\mathscr{V}$ are two good covers of $M$, then $Im(\gamma_\msU)=Im(\gamma_{\mathscr{V}})$.
\end{cor}

\begin{proof}
See Bott and Tu \cite{bott}
\end{proof}

\begin{rem}
Let $\omega\in\Omega^2(M)$ be a closed 2-form and $\{U_i\}_{i\in I}$ a good cover of $M$. We can construct a \v{C}ech 2-cocycle from $\omega$ as follows.

Since $\omega$ is closed and since $U_i$ is simply connected, there exists $\alpha_i\in\Omega^1(U_i)$ such that on $U_i$
\bge
d\alpha_i=\omega.
\ene
It's easy to see then that if $U_i\cap U_j\neq\emptyset$, then $d(\alpha_i-\alpha_j)=0$. Hence, there exists $f\in C^\infty(U_i\cap U_j)$ such that
\bge
\alpha_i-\alpha_j=d f_{ij}.
\ene
Doing a similar computation on $U_{ijk}$ can conclude
\bge
d(-f_{ij}+f_{ik}-f_{jk})=0
\ene
And hence there exists a constant function $\mu_{ijk}$ on $U_i\cap U_j\cap U_k$ such that
\bge
-f_{ij}+f_{ik}-f_{jk}=\mu_{ijk}.
\ene

It's an easy compuation to see that $\mu_{ijk}$ is indeed a \v{C}ech 2-cocycle. 
\end{rem}

\begin{defs}
Call the $\mu_{ijk}$ constructed in the above remark the \v{C}ech cocycle of $[\omega]$ with respect to the cover $\{U_i\}_{i\in I}$.
\end{defs}

\begin{thm}
$[\omega]$ is integral $\iff$ the \v{C}ech cocycles of $[\omega]$, $\mu_{ijk}$, are constant integer functions for all $i,j,k$ and any choice of good cover.
\end{thm}

\begin{proof}
See Kostant \cite{kostant}
\end{proof}

\subsection{Main Theorem}
At last, we have built up enough theory to prove the main result. The proof here is the one that can be found in Kostant's original paper \cite{kostant}.

\begin{thm}
Let $M$ be a smooth manifold and $\omega\in \Omega^2(M)$ a closed 2-form. Then there exists a Hermitian line bundle with compatible covariant derivative $(L,\bra,\ket,\nabla)$ over $M$ with curvature $2\pi i\omega$ $\iff$ $[\omega]$ is integral.
\end{thm}

\begin{proof}
\begin{itemize}
    \item[$\Rightarrow$:] Suppose $(L,\bra,\ket,\nabla)$ is Hermitian line bundle with compatible covariant derivative over $M$ with $2\pi i\omega$ the curvature.
    
    If $\varphi$ is the corresponding connection 1-form, we have $d\varphi=\pi^*\omega$, where $\pi:L\ra M$ is the line bundle projection.
    
    Let $\{(U_i,s_i)\}_{i\in I}$ be a local system, $g_{ij}$ the transition functions. Since we have a Hermitian structure, may pick the $s_i$ so that $\bra s_i,s_i\ket=1$ and hence $|g_{ij}|=1$. 
    
    Define $\alpha_i=s_i^*\varphi$. Then since $\pi\circ s_i=\id_{U_i}$, we have on $U_i$
    \bge
    d\alpha_i=d(s_i^*\varphi)=s_i^*d\varphi=s_i^*\pi^*\omega=\omega.
    \ene
    Furthermore, since $s_i=g_{ij}s_j$ and since for all local vector fields $\xi$ we have
    \bge
    \nabla_\xi s_i=2\pi i\alpha_i(\xi)s_i,\quad \nabla_\xi s_j=2\pi i\alpha_j(\xi)s_j,
    \ene
    one can easily show that
    \bge
    \alpha_i-\alpha_j=\frac{1}{2\pi i}\frac{dg_{ij}}{g_{ij}}.
    \ene

    Now, recall that if $U_i\cap U_j\cap U_k\neq\emptyset$, then $g_{ik}=g_{ij}g_{jk}$ on $U_i\cap U_j\cap U_k$. Furthermore, by construction we have $|g_{ij}|=1$. Thus, choosing a branch of logarithm $\log(g_{ij})$  and defining
    \bge
    f_{ij}=\frac{1}{2\pi i}\log(g_{ij}),
    \ene
    we obtain there is some integer valued $\mu_{ijk}:U_i\cap U_j\cap U_k\ra\bZ$ such that
    \bge
    -f_{ij}+f_{ik}-f_{jk}=\mu_{ijk}.
    \ene
    Since $\mu_{ijk}$ is continuous and $U_i\cap U_j\cap U_k$ is connected, we must have that $\mu_{ijk}$ is constant.
    
    Furthermore, since we now have
    \bge
    \alpha_i-\alpha_j=d f_{ij}
    \ene
    we conclude that $\mu_{ijk}$ is a Cech cocycle of $[\omega]$ with respect to the cover $\{U_i\}_{i\in I}$. Since the $\mu_{ijk}$ are constant integer functions, we conclude by Theorem 7 that $[\omega]$ is integral.
    
    \item[$\Leftarrow$:] Suppose $[\omega]$ is integral. Let $\{U_i\}_{i\in I}$ be a good cover. On $U_i$ there exists $\alpha_i\in\Omega^1(U_i)$ such that
    \bge
    \omega=d\alpha_i.
    \ene
    On $U_i\cap U_j$, there exists $f_{ij}\in C^\infty(U_i\cap U_j)$ such that
    \bge
    \alpha_i-\alpha_j=df_{ij}.
    \ene
    Define $g_{ij}=e^{2\pi i f_{ij}}$. Then the $g_{ij}$ satisfy the cocycle condition in equation (4) and
    \bge
    \alpha_i-\alpha_j=\frac{1}{2\pi i}\frac{dg_{ij}}{g_{ij}}.
    \ene
    Hence by Theorem 5, there exists a line bundle with covariant derivative $(L,\nabla)$ over $M$ such that the $g_{ij}$ are the transition functions of the local system $\{(U_i,s_i)\}_{i\in I}$. Furthermore, if $\varphi$ is the corresponding connection 1-form on $L^+$, we have
    \bge
    \alpha_i=s_i^*\varphi
    \ene
    Note that $s_i^* d\varphi=d\alpha_i=\omega$. Hence, $2\pi i\omega$ is the curvature form for $\nabla$.
    
    All that's left is to demonstrate a Hermitian form that is compatible with $\nabla$. To do this, let $i\in I$ and $z,w\in U_i$. Define
    \bge
    \bra z,w\ket_i:=\frac{z}{s_i}\overline{\frac{w}{s_i}}.
    \ene
    We show the $\bra,\ket_i$ agree on overlaps and hence define a Hermitian form on $L$. 
    
    Since $|g_{ij}|=1$, we have if $z,w\in \pi^{-1}(U_i\cap U_j)$
    \bge
    \bra z,w\ket_i:=\frac{z}{s_i}\overline{\frac{w}{s_i}}=\frac{1}{|g_{ij}|^2}\frac{z}{s_j}\overline{\frac{w}{s_j}}=\bra z,w\ket_j.
    \ene
    
    Hence the $\bra,\ket_i$ piece together to a globally defined Hermitian form $\bra,\ket$. We now just have to show that $\nabla$ and $\bra,\ket$ are compatible. Suffices to show for all local vector fields $\xi$
    \begin{equation}
    \bra \nabla_\xi s_i,s_i\ket+\bra s_i,\nabla_\xi s_i\ket=0.
    \end{equation}
    
    To see this, let $x\in U_i$ and let $r,t$ be local sections of $L$ in a neighbourhood of $x$. Then there exists smooth complex valued functions $F,G$ near $x$ such that
    \bge
    r=Fs_i\quad t=Gs_i.
    \ene
    Then given that the equation above is true, we have for any local vector field $\xi$ defined around $x$
    \bge
    \xi_x\bra r,t\ket=\xi_x(F\overline{G}\bra s_i,s_i\ket)=\xi_x(F\overline{G}).
    \ene
    On the otherhand,
    \begin{align*}
    \bra \nabla_\xi r,t\ket+\bra r,\nabla_\xi t\ket&=\bra (\xi F)s_i+F\nabla_\xi s_i,Gs_i\ket +\bra Fs_i, (\xi G)s_i+G\nabla_\xi s_i,t\ket\\
    &=(\xi F+\overline{\xi G})\bra s_i,s_i\ket+F\overline{G}(\bra \nabla_\xi s_i,s_i\ket+\bra s_i,\nabla_\xi s_i\ket)\\
    &=\xi(F\overline{G}).
    \end{align*}
    
    Now all that's left is to show equation (5) holds. Let $i\in I$, then we compute,
    \begin{align*}
    \bra \nabla_\xi s_i,s_i\ket&=\bra 2\pi i\alpha_i(\xi) s_i,s_i\ket=2\pi i\alpha_i(\xi).
    \end{align*}
    Hence, since $\alpha_i$ is real valued
    \bge
    \bra \nabla_\xi s_i,s_i\ket+\bra s_i,\nabla_\xi s_i\ket=2\pi i-2\pi i=0.
    \ene
    This completes the proof and this paper.
\end{itemize}
\end{proof}

\end{document}